\documentclass[a4paper]{article}
\usepackage{amsthm,amsmath,amssymb}
\usepackage{enumerate,enumitem}
\usepackage{graphicx}
\usepackage{xfrac}
\usepackage{ulem}
\usepackage{mathtools}
\usepackage{ wasysym }%Bowtie%gemini
\usepackage{float}
\usepackage{tabularx,booktabs,multirow}
\usepackage{hyperref}
\usepackage{microtype}
\newcolumntype{C}{>{\centering\arraybackslash}X}
\newcolumntype{D}{>{\centering\arraybackslash}X}

\DeclareGraphicsExtensions{.pdf,.png,.eps}
\graphicspath{{figs/}}

\newtheorem{theorem}{Theorem}
\newtheorem{lemma}[theorem]{Lemma}

\newtheorem{construction}[theorem]{Construction}

\newtheorem{corollary}[theorem]{Corollary}

\newtheorem*{claim*}{Claim}

\theoremstyle{remark}

\newcommand{\cL}{\ensuremath{\mathcal{L}}}

\newcommand{\cA}{\ensuremath{\mathcal{A}}}

\newcommand{\cB}{\ensuremath{\mathcal{B}}}

\newcommand{\cD}{\ensuremath{\mathcal{D}}}

\newcommand{\cF}{\ensuremath{\mathcal{F}}}

\newcommand{\cP}{\ensuremath{\mathcal{P}}}
\newcommand{\cH}{\ensuremath{\mathcal{H}}}
\newcommand{\cS}{\ensuremath{\mathcal{S}}}

\newcommand{\cT}{\ensuremath{\mathcal{T}}}
\newcommand{\N}{\ensuremath{\mathbb{N}}}
\newcommand{\cN}{\ensuremath{\mathcal{N}}}
\newcommand{\cK}{\ensuremath{\mathcal{K}}}

\newcommand{\cU}{\ensuremath{\mathcal{U}}}
\newcommand{\cV}{\ensuremath{\mathcal{V}}}
\newcommand{\cW}{\ensuremath{\mathcal{W}}}
\newcommand{\bX}{\ensuremath{\mathbf{X}}}
\newcommand{\bY}{\ensuremath{\mathbf{Y}}}
\newcommand{\bZ}{\ensuremath{\mathbf{Z}}}

\newcommand{\EEE}{\ensuremath{\mathbb{E}}}

\newcommand{\QQ}{\ensuremath{\mathcal{Q}}}
\newcommand{\PPP}{\ensuremath{\mathbb{P}}}

\makeatletter
\newcommand\olt{\mathbin{\!\mathpalette\make@circled{\hspace{-1.05pt}\vcenter{\hbox{\scalebox{0.91}{$\m@th <$}}}}}}
\newcommand\opl{\mathbin{\!\mathpalette\make@circled{\vcenter{\hbox{\scalebox{0.65}{$\m@th ||$}}}}}}
\newcommand{\make@circled}[2]{\raisebox{1pt}{
  \ooalign{$\m@th#1\smallbigcirc{#1}$\cr\hidewidth$\m@th#1#2$\hidewidth\cr}}%
}
\newcommand{\smallbigcirc}[1]{%
  \vcenter{\hbox{\scalebox{0.76}{$\m@th#1\bigcirc$}}}%
}
\makeatother

\usepackage[usenames,dvipsnames]{xcolor}

\begin{document}

\title{Erd\H{o}s-Hajnal problems for posets}
\author{Christian Winter \footnote{Karlsruhe Institute of Technology, Karlsruhe, Germany. E-mail: \textit{christian.winter@kit.edu}} }
\maketitle

\begin{abstract}
We say that a poset $(Q,\le_{Q})$ contains an induced copy of a poset $(P,\le_P)$ if there is an injective function $\phi\colon P\to Q$ such that for every two $X,Y\in P$,\;\;$X\le_P Y$ if and only if $\phi(X)\le_Q \phi(Y)$. We denote the Boolean lattice $(2^{[n]},\subseteq)$ by $Q_n$.
%A poset $(P_2,\le_2)$ is an induced subposet of a poset $(P_1,\le_1)$ if $P_2\subseteq P_1$ and for every two $X,Y\in P_2$, $X\le_2 Y$ if and only if $X\le_1 Y$.
Given a fixed $2$-coloring $c$ of a poset $P$, the poset Erd\H{o}s-Hajnal number of this colored poset is the smallest integer $N$ such that every $2$-coloring of the Boolean lattice $Q_N$ contains an induced copy of $P$ colored as in $c$, or a monochromatic induced copy of $Q_n$.
We present bounds on the poset Erd\H{o}s-Hajnal number of general colored posets, antichains, chains, and small Boolean lattices.

Let the poset Ramsey number $R(Q_n,Q_n)$ be the least $N$ such that every $2$-coloring of $Q_N$ contains a monochromatic induced copy of $Q_n$.
As a corollary, we show that $R(Q_n,Q_n)> 2.02n$, improving on the best known lower bound $2n+1$ by Cox and Stolee \cite{CS}.
\end{abstract}

\section{Introduction}
The classic question in Ramsey theory is to quantify the size of a host structure such that in any coloring of its elements, a large monochromatic substructure exists.
In the setting of graphs, Erd\H{o}s and Hajnal \cite{EH} introduced a related problem: 
Given a fixed graph $H$ edge-colored with colors blue and red, 
determine the minimal order of a complete graph such that any blue/red coloring of its edges contains 
a subgraph isomorphic to $H$ with a matching color pattern, or a monochromatic complete graph on $n$ vertices.
The well-known Erd\H{o}s-Hajnal conjecture states that the answer to the above problem is at most $n^{c(H)}$ where $c(H)$ is a constant, depending on $H$.
This conjecture is wide-open for most graphs $H$. For more details, we refer to a survey by Chudnovsky \cite{Chudnovsky} and other recent results, e.g., \cite{MSZ, Weber, NSS}.
In this paper, we propose a similar concept for \textit{posets}.
\\

A \textit{poset} is a set $P$ equipped with a binary relation $\le_P$ which is transitive, reflexive, and anti-symmetric.
The \textit{Boolean lattice} $Q_n$ of \textit{dimension} $n $ is the poset consisting of all subsets of an $n$-element ground set, ordered by the inclusion relation $\subseteq$.
The elements of $P$ are usually referred to as \textit{vertices}.
A \textit{colored poset} is a pair $(P,c_P)$, where $P$ is a poset and $c_P\colon P\to\{\text{blue, red}\}$ is a blue/red coloring of the vertices of $P$.
If a poset $P$ has a fixed coloring $c_P$, we usually write $\dot P$ instead of $(P,c_P)$. The \textit{size} of a colored poset $\dot P$ is the size of the underlying poset $P$.
Occasionally, we specify the assigned coloring using an additional superscript. 
In particular, the poset $P$ which is colored monochromatically blue is denoted by ${\dot P^{(b)}}$. In this case, we say that $\dot P$ is \textit{blue}. 
Similarly, we refer to a poset $P$ colored monochromatically red as ${\dot P^{(r)}}$ and say that $\dot P$ is \textit{red}.
%A poset $\dot P$ is \textit{monochromatic} if it is either blue or red.
\medskip

A poset $P$ is an \textit{induced subposet}, or \textit{subposet} for short, of a poset $Q$ if $P\subseteq Q$ and for any two $X,Y\in P$, $X\le_P Y$ if and only if $X\le_Q Y$.
A \textit{copy} of a poset $P$ in $Q$ is an induced subposet $P'$ of $Q$ that is isomorphic to $P$.
Equivalently, a copy is the image of an \textit{embedding} $\phi\colon P\to Q$, 
i.e., a function such that for every $X,Y\in P$,\:\:$X\le_{P} Y$ if and only if $\phi(X)\le_{Q}\phi(Y)$.
Given a fixed blue/red coloring of $Q$, a \textit{colored copy}, or \textit{copy} for short, of a colored poset $\dot P$ in $Q$ is a copy $P'$ of $P$ in $Q$ such that each vertex $Z\in P'$ has the same color in $Q$ as its corresponding vertex in $\dot P$.
For any fixed colored poset $\dot P$, a blue/red coloring of $Q$ is \textit{$\dot P$-free} if it contains no colored copy of $\dot P$.
\\

Extending the classic definition of Ramsey numbers for graphs, Axenovich and Walzer \cite{AW} introduced the \textit{poset Ramsey number} $R(P_1,P_2)$ of posets $P_1$ and $P_2$, defined as
the smallest $N\in\N$ for which every coloring of $Q_N$ contains a copy of $\dot P_1^{(b)}$ or $\dot P_2^{(r)}$.
%$$R(P_1,P_2)=\min\{N\in\N:~ \text{every coloring of }Q_N\text{ contains a copy of }\dot P_1^{(b)}\text{ or }\dot P_2^{(r)}\}.$$
A central question in this setting is to determine $R(Q_n,Q_n)$, where the best upper bound is currently $R(Q_n,Q_n)\le n^2-\Theta(n\log n)$, 
see listed chronologically Walzer \cite{Walzer}, Axenovich and Walzer \cite{AW}, Lu and Thompson \cite{LT}, Axenovich and Winter \cite{QnQn}.
The best known lower bound is $R(Q_n,Q_n)\ge 2n+1$ by Cox and Stolee \cite{CS} who improved the trivial lower bound $2n$ using a probabilistic construction for $n\ge 13$.
Later, Bohman and Peng \cite{BP} gave an explicit construction proving the same bound for $n\ge 3$.
Exact bounds are only known for $n\le 3$;
Axenovich and Walzer \cite{AW} showed that $R(Q_2,Q_2)=4$, and Falgas-Ravry, Markstr\"om, Treglown and Zhao \cite{FMTZ} proved that $R(Q_3,Q_3)=7$.
\\

For $n\in\N$, the \textit{poset Erd\H{o}s-Hajnal number} $\widetilde{R}(\dot P,Q_n)$ of a colored poset $\dot P$ is the smallest $N\in\N$ such that every blue/red coloring of $Q_N$ contains a copy of $\dot P$, $\dot Q_n^{(b)}$, or $\dot Q_n^{(r)}$.
In other words, $\widetilde{R}(\dot P,Q_n)$ is the minimal $N$ such that any $\dot P$-free blue/red coloring of $Q_N$ contains a monochromatic copy of $Q_n$.
In this paper, we study the poset Erd\H{o}s-Hajnal number $\widetilde{R}(\dot P,Q_n)$ for a fixed colored poset $\dot P$, while $n$ is usually large. 

If $\dot P$ is monochromatic, then $\widetilde{R}(\dot P,Q_n)=R(P,Q_n)$ for large $n$.
This poset Ramsey setting has been addressed in multiple articles, see listed chronologically Lu and Thompson \cite{LT}, Gr\'osz, Methuku, and Tompkins \cite{GMT}, Axenovich and Winter \cite{QnV, QnK, QnN, QnPA}.
In this paper, we focus on colored posets $\dot P$ in which both colors occur. 

We say that $\dot P$ is \textit{diverse} if it contains two comparable vertices of distinct color. Otherwise, $\dot P$ is said to be \textit{non-diverse}.
Our first results provide general bounds for the poset Erd\H{o}s-Hajnal number of diverse and non-diverse $\dot P$, respectively.
The \textit{height} $h(P)$ of a poset $P$ is the size of the largest chain in $P$, and the $2$-dimension $\dim_2(P)$ of $P$ is the smallest $N$ such that $Q_N$ contains a copy of $P$.
It is a basic observation that the $2$-dimension is well-defined for any poset $P$.
%A general bound on the Erd\H{o}s-Hajnal number of diverse posets is easy to see.

\begin{theorem}\label{thm:EHgenFUL}
Let $\dot P$ be a diverse colored poset. Let $n\in\N$.
Then $$2n\le \widetilde{R}(\dot P,Q_n)\le h(P)n+\dim_2(P).$$
\end{theorem}

\noindent This bound has a straightforward proof. Say that $\dot P$ contains a red vertex which is larger than some blue vertex.
The lower bound is obtained from a layered coloring of $Q_{2n-1}$, in which vertices $Z$ with $|Z|\le n-1$ are colored in red, and vertices $Z$ such that $|Z|\ge n$ in blue.
The upper bound follows from Lemma 3 in Axenovich and Walzer \cite{AW}. We omit the details.

We define the \textit{parallel composition} $\dot P_1 \opl \dot P_2$ of two colored posets $\dot P_1$ and $\dot P_2$ as the colored poset consisting of a copy of $\dot P_1$ and a copy of $\dot P_2$ that are \textit{parallel}, i.e., element-wise incomparable.
Observe that a colored poset $\dot P$ is non-diverse if and only if $P$ has subposets $P_b$ and $P_r$ such that $\dot P=\dot P_b^{(b)}\opl \dot P_r^{(r)}$.

\begin{theorem}\label{thm:EHgenNON}
Let $\dot P$ be a non-diverse poset. Let $P_r$ and $P_b$ such that $\dot P=\dot P_b^{(b)}\opl \dot P_r^{(r)}$.
Let $n\in\N$ with $n\ge \max\{\dim_2(P_b), \dim_2(P_r)\}$.
%Let ${\dot\cB^{(b)}}$ be its blue induced subposet and ${\dot\cR^{(r)}}$ its red induced subposet, i.e., $\dot P=\dot\cB^{(b)}\opl \dot\cR^{(r)}$. 
Then 
$$\max\{R(P_b,Q_n),R(P_r,Q_n)\}\le \widetilde{R}(\dot P,Q_n)\le \max\{R(P_b,Q_n),R(P_r,Q_n)\} +2.$$
\end{theorem}

\noindent 
For a fixed poset $P$, Axenovich and Walzer~\cite{AW} showed that $R(P,Q_n)=O(n)$, so Theorems~\ref{thm:EHgenFUL} and \ref{thm:EHgenNON} imply that $\widetilde{R}(\dot P,Q_n)=O(n)$ for any fixed $\dot P$. 
Recall that without forbidding $\dot P$, the best known upper bound is $R(Q_n,Q_n)=O(n^2)$, which is quadratic rather than linear. 
In that respect, our results confirm an analogue of the Erd\H{o}s-Hajnal conjecture for posets.

The simplest non-diverse colored poset is an \textit{antichain} $A_t$, i.e., a poset consisting of $t$ pairwise incomparable vertices. 
We precisely determine the Erd\H{o}s-Hajnal number for antichains.

\begin{theorem}\label{thm:EHantichain}
Let $\dot A$ be a non-monochromatic antichain on at least $2$ vertices. Let $n$ be sufficiently large.
If there are no three vertices of the same color in $\dot A$, then $\widetilde{R}(\dot A, Q_n)=n+2$.
Otherwise, $\widetilde{R}(\dot A, Q_n)=n+3$.
\end{theorem}

\noindent In particular, $\widetilde{R}({\dot A_2^{(b)}}\opl {\dot A_2^{(r)}},Q_n)=n+2=R(A_2,Q_n)$, and $\widetilde{R}({\dot A_1^{(b)}}\opl {\dot A_1^{(r)}},Q_n)=n+2=R(A_1,Q_n)+2$, 
which attain the lower and upper bound in Theorem \ref{thm:EHgenNON}, respectively.
We do not attempt to determine the smallest $n$ for which this bound holds. In our proof, we require $\log \log \log n=\Omega(|A|)$.
\\

A \textit{chain} $C_t$ is a poset on $t$ pairwise comparable vertices.
For colored chains, we introduce two specific colorings.
The \textit{red-alternating chain} $\dot C_t^{(rbr)}$ is the chain~$C_t$ whose vertices are colored alternatingly in red and blue, such that the minimal vertex is red, see Figure \ref{fig:coloredQ2} for an illustration.
Similarly, the \textit {blue-alternating chain} $\dot C_t^{(brb)}$ is the chain~$C_t$ colored alternatingly, but the minimal vertex is blue.

Given a colored chain $\dot C$, let $\lambda(\dot C)$ be the largest integer $\ell$ such that $\dot C$ contains a copy of $\dot C_\ell^{(rbr)}$ or $\dot C_\ell^{(brb)}$.
Theorem \ref{thm:EHgenFUL} implies that $\widetilde{R}(\dot C,Q_n)$ is linear in terms of $n$.
Our next result shows that the poset Erd\H{o}s-Hajnal number of any colored chain~$\dot C$ is determined by the poset Erd\H{o}s-Hajnal number of an alternating chain, up to an additive term independent of~$n$.

\begin{theorem}\label{thm:EHreduction}
Let $n\in\N$. Let $\dot C_t$ be a colored chain of length $t$, and let $\lambda=\lambda(\dot C_t)$. Then
$$\widetilde{R}(\dot C^{(rbr)}_{\lambda},Q_n)\le \widetilde{R}(\dot C_t,Q_n)\le \widetilde{R}(\dot C^{(rbr)}_{\lambda},Q_n)+t-\lambda.$$
\end{theorem} 

\noindent For alternating chains, we give the following bounds.
\smallskip
\begin{theorem}\label{thm:EHchain}
For every $n$,\:\:$\widetilde{R}(\dot C^{(rbr)}_{2},Q_n)=\widetilde{R}(\dot C^{(rbr)}_{3},Q_n)=2n$. 
For $t\ge 4$ and sufficiently large $n$, $$2.02n< \widetilde{R}(\dot C^{(rbr)}_{t},Q_n)\le (t-1)n.$$
\end{theorem}

\noindent The lower bound on $\widetilde{R}(\dot C^{(rbr)}_{t},Q_n)$ shows the existence of a blue/red coloring of $Q_{2.02n}$ with no monochromatic $Q_n$. 
%\noindent As an immediate corollary, Theorem \ref{thm:EHchain} provides an improved lower bound on the poset Ramsey number of $Q_n$ and $Q_n$.

\begin{corollary}\label{cor:QnQnLB}
For sufficiently large $n$,\:\:$R(Q_n,Q_n)> 2.02n$.
\end{corollary}

\begin{figure}[h]
\centering\label{fig:coloredQ2}
\includegraphics[scale=0.62]{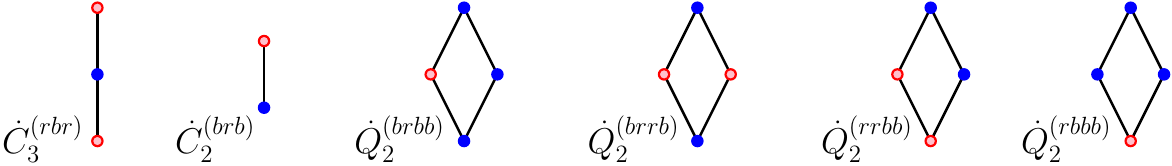}
\caption{Alternating chains and non-monochromatic colorings of $Q_2$.}
\end{figure}

Finally, we analyze the poset Erd\H{o}s-Hajnal number of small colored Boolean lattices.
Up to permutation of colors, the only non-monochromatic blue/red coloring of $Q_1$ is $\dot C^{(rbr)}_{2}$.
Theorem \ref{thm:EHchain} shows that $\widetilde{R}(\dot C^{(rbr)}_{2},Q_n)=2n$.
Moreover, we give bounds on $\widetilde{R}(\dot Q_2,Q_n)$ for every non-monochromatic blue/red coloring of~$Q_2$.
Up to symmetry and permutation of colors, the four non-monochromatic $Q_2$ are $\dot Q_2^{(brbb)}$, $\dot Q_2^{(brrb)}$, $\dot Q_2^{(rrbb)}$, and $\dot Q_2^{(rbbb)}$, each with the respective coloring as illustrated in Figure \ref{fig:coloredQ2}.

\begin{theorem}\label{thm:EHbool}
For every $n\in\N$,\:\:$\widetilde{R}(\dot Q_2^{(brbb)},Q_n)=\widetilde{R}(\dot Q_2^{(brrb)},Q_n)=\widetilde{R}(\dot Q_2^{(rrbb)},Q_n)=2n$, and
$2n\le \widetilde{R}(\dot Q_2^{(rbbb)},Q_n)\le 2n + O\big(\tfrac{n}{\log n}\big).$
\end{theorem}

We omit the proof of Theorem \ref{thm:EHbool} here. For a full proof, the reader is referred to \cite{Winter-diss}.
\\

The article is structured as follows.
In Section \ref{sec:EHgen}, we introduce notation and preliminary lemmas.
In Section \ref{sec:EHnoncolorful}, we study non-diverse posets and prove Theorems \ref{thm:EHgenNON} and \ref{thm:EHantichain}.
Afterwards, in Section \ref{sec:EHchain}, we focus on chains and present proofs for Theorems \ref{thm:EHreduction} and \ref{thm:EHchain}.
%In the final Section \ref{sec:EHbool}, we verify Theorem \ref{thm:EHbool}.

%%%%%%%%%%%%%%%%%%%%%%%%%%%%%%%%%%%%%%%%%%%%%%%%%%%%%%%%%%%%%%%%%%%%%%%
%%%
%%%%%%%%%%%%%%%%%%%%%%%%%%%%%%%%%%%%%%%%%%%%%%%%%%%%%%%%%%%%%%%%%%%%%%%

\section{Preliminaries}\label{sec:EHgen}

\subsection{Basic Notation}

Let $[n]=\{1,\dots,n\}$ for every $n\in\N$. In this paper `$\log$' always refers to the logarithm with base $2$.
We omit floors and ceilings where appropriate.
We denote by $\QQ(\bZ)$ the Boolean lattice with ground set $\bZ$, i.e., the poset of subsets of $\bZ$ ordered by inclusion.
For $\ell\in\{0,\dots,|\bZ|\}$, \textit{layer $\ell$} of $\QQ(\bZ)$ refers to the subposet $\{Z\in\QQ(\bZ): ~ |Z|=\ell\}$.
Note that every layer of the Boolean lattice is an antichain.
Given a Boolean lattice $\QQ$ and vertices $A,B\in\QQ$ with $A\subseteq B$, 
the \textit{sub-Boolean lattice}, or \textit{sublattice} for short, between $A$ and $B$ is
$$\QQ\big|_A^B =\{X\in\QQ : A\subseteq X \subseteq B\}.$$
This subposet is isomorphic to a Boolean lattice of dimension $|B|-|A|$.
Note that a copy of a Boolean lattice in $\QQ$ is not necessarily a sublattice.

\subsection{Red Boolean lattice versus blue chain}

Let $\bX$ and $\bY$ be non-empty, disjoint sets and let $k=|\bY|$.
We denote a linear ordering $\tau$ of $\bY$ where $y_1<_\tau y_2<_\tau \dots <_\tau y_k$ by a sequence $\tau=(y_1,\dots,y_k)$.
Fix a linear ordering $\tau=(y_1,\dots,y_k)$ of $\bY$.
A \textit{$\bY$-chain} corresponding to $\tau$ is a $(k+1)$-element chain in the Boolean lattice $\QQ(\bX\cup\bY)$ on vertices
$$X_0\cup \varnothing,X_1\cup\{y_1\},X_2\cup\{y_1,y_2\},\dots,X_{k}\cup \bY,$$
 where $X_0 \subseteq X_1 \subseteq\dots \subseteq X_k\subseteq \bX$. 
%Note that $\bY$-chains corresponding to distinct linear orderings of $\bY$ are distinct.
The following result was proved implicitly by Chen, Cheng, Li and Liu, see Theorem 15 of \cite{CCLL}, as well as by Gr\'osz, Methuku and Tompkins, see Claim 3 of  \cite{GMT}. For an alternative proof, see \cite{QnV}.

\begin{lemma}[\cite{CCLL}]\label{lem:chain}\label{chain_lem}
Let $\bX$ and $\bY$ be disjoint sets with $|\bX|=n$ and $|\bY|=k$. Let $\QQ(\bX\cup\bY)$ be a colored Boolean lattice.
Fix a linear ordering $\tau=(y_1,\dots,y_k)$ of $\bY$. Then there exists either a red copy of $Q_n$, or a blue $\bY$-chain corresponding to $\tau$ in $\QQ(\bX\cup\bY)$.
\end{lemma}

The following corollary was shown independently by Axenovich and Walzer \cite{AW}.

\begin{corollary}[\cite{AW}]\label{cor:chain}
Let $n$ and $k$ be positive integers. Let $\QQ$ be a colored Boolean lattice of dimension $n+k$. 
Then $\QQ$ contains a red copy of $Q_n$ or a blue chain of length $k+1$.
%In particular, for any poset $P\subseteq\QQ([N])$, there is a copy of a Boolean lattice of dimension $N- h(P)$ in $\QQ([N])$ which contains no vertex from $P$.
\end{corollary}

\subsection{Embedding of a Boolean lattice}

Recall that a \textit{copy} of a poset $P$ in another poset $Q$ is defined as a subposet of $Q$ which is isomorphic to $P$.
Equivalently, a copy of $P$ in $Q$ is the image of an \textit{embedding} $\phi\colon P\to Q$, 
i.e., an injective function such that for any two $X,Y\in P$, $X\le_P Y$ if and only if $\phi(X)\le_Q \phi(Y)$.
%Let $\bX$ and $\bY$ be non-empty, disjoint sets.
%We say that an embedding $\phi\colon \QQ(\bX)\to \QQ(\bX\cup\bY)$ is \textit{$\bX$-good} if $\phi(X)\cap \bX=X$ for all $X\subseteq \bX$.
Axenovich and Walzer \cite{AW} showed that every embedding of a small Boolean lattice into a larger Boolean lattice has the following nice property, see Theorem 8 of \cite{AW}.

\begin{lemma}[\cite{AW}]\label{lem:embed}\label{embed_lem} %[Embedding Lemma]
Let $n\in\N$ and let $\bZ$ be a set with $|\bZ|\ge n$. If there is an embedding $\phi\colon Q_n\to \QQ(\bZ)$,
then there exist a subset $\bX\subseteq\bZ$ with $|\bX|=n$, and an embedding $\phi'\colon \QQ(\bX)\to \QQ(\bZ)$ with the same image as $\phi$ such that $\phi'(X)\cap \bX=X$ for all $X\subseteq \bX$.
\end{lemma}

%%%%%%%%%%%%%%%%%%%%%%%%%%%%%%%%%%%%%%%%%%%%%%%%%%%%%%%%%%%%%%%%%%%%%%%
%%%
%%%%%%%%%%%%%%%%%%%%%%%%%%%%%%%%%%%%%%%%%%%%%%%%%%%%%%%%%%%%%%%%%%%%%%%

\section{Forbidden non-diverse colored posets}\label{sec:EHnoncolorful}

\begin{proof}[Proof of Theorem \ref{thm:EHgenNON}]
For the lower bound, note that $P_b\subseteq Q_n$ by the choice of $n$. Thus, $R(P_b,Q_n)\le \widetilde{R}(\dot P^{(b)}_b,Q_n) \le \widetilde{R}(\dot P,Q_n)$.
A similar argument shows that $R(P_r,Q_n)\le  \widetilde{R}(\dot P,Q_n)$.

To establish the upper bound, let $m=\max\{R(P_b,Q_n), R(P_r,Q_n)\}$ and $N=m+2$.
Consider an arbitrary blue/red coloring of the Boolean lattice $\QQ=\QQ([N])$ which contains no monochromatic copy of $Q_n$.
We shall show that this coloring contains a copy of $\dot P$.
Note that the sublattices $\QQ\big|_{\{1\}}^{[N]\setminus\{2\}}$ and $\QQ\big|_{\{2\}}^{[N]\setminus\{1\}}$ are parallel.
The sublattice $\QQ\big|_{\{1\}}^{[N]\setminus\{2\}}$ is isomorphic to a Boolean lattice of dimension $N-2=m\ge R(P_b,Q_n)$, thus it contains a blue copy of $P_b$. 
Similarly, $\QQ\big|_{\{2\}}^{[N]\setminus\{1\}}$ contains a red copy of $P_r$.
By combining these two subposets, we obtain a copy of $\dot P$.
\end{proof}

Theorem \ref{thm:EHantichain} is a consequence of the following three lemmas. %Recall that $C_t$ is the chain on $t$ vertices and $A_t$ is the antichain on $t$ vertices.

\begin{lemma}\label{lem:EHantichain1}
For every $1\le s\le t< n$,\:\:$\widetilde{R}({\dot C_t^{(b)}}\opl {\dot C_s^{(r)}},Q_n)=n+t+1$. 
\end{lemma}

\begin{proof}
The upper bound $\widetilde{R}({\dot C_t^{(b)}}\opl {\dot C_s^{(r)}},Q_n)\le R(C_t,Q_n)+2=n+t+1$ is implied by Theorem \ref{thm:EHgenNON} and Corollary \ref{cor:chain}.
We shall prove the lower bound by constructing a layered coloring of $\QQ([n+t])$ that contains neither a copy of $\dot C_t^{(b)}\opl \dot C_s^{(r)}$ nor a monochromatic copy of $Q_n$.
Assign the color blue to the two vertices $\varnothing$ and $[n+t]$ as well as to all vertices in $t-1$ arbitrarily chosen additional layers. Color all remaining vertices in red.
There are $t+1\le n$ blue layers and $n$ red layers in our coloring.
Since $Q_n$ contains a chain on $n+1$ vertices, there is no monochromatic copy of $Q_n$.
Next, assume towards a contradiction that there exists a copy $\dot \cP$ of $\dot C_t^{(b)}\opl \dot C_s^{(r)}$.
The subposet $\dot \cP$ contains $t$ pairwise comparable blue vertices. 
Since there are $t+1$ blue layers in our coloring, either $\varnothing$ or $[n+t]$ are contained in $\dot \cP$. 
Both of these vertices are comparable to every other vertex of the copy of $\dot \cP$.
However, every blue vertex of $\dot\cP$ is incomparable to every red vertex of $\dot\cP$, a contradiction.
\end{proof}

\begin{lemma}\label{lem:EHantichain2}
For $n\ge 3$,\:\:$\widetilde{R}({\dot A_2^{(b)}}\opl {\dot A_2^{(r)}},Q_n)=n+2$.
\end{lemma}
\begin{proof}
The lower bound $\widetilde{R}({\dot A_2^{(b)}}\opl {\dot A_2^{(r)}},Q_n)\ge R(A_2,Q_n)=n+2$ follows from Theorem~\ref{thm:EHgenNON} and the fact that $R(A_2,Q_n)= n+2$ by Theorem 5 in \cite{QnPA}.
For the upper bound, let $N=n+2$ and fix an arbitrary blue/red coloring of the Boolean lattice $\QQ=\QQ([N])$.
We shall show that there is either a colored copy of ${\dot A_2^{(b)}}\opl {\dot A_2^{(r)}}$ or a monochromatic copy of $Q_n$.

We say that a layer $\{Z\in\QQ: ~ |Z|=i\}$, $i\in\{1,\dots,n+1\}$, is \textit{almost red} if it contains at most one blue vertex, and \textit{almost blue} if it contains at most one red vertex.
We can suppose that every layer $i$, where $i\in\{1,\dots,n+1\}$, is almost red or almost blue; otherwise, such a layer contains a copy of ${\dot A_2^{(b)}}\opl {\dot A_2^{(r)}}$.
If there are consecutive layers $i$ and $i+1$, $i\in \{1,\dots, n\}$, such that one of them is almost red and one is almost blue,
then it is straightforward to find a copy of ${\dot A_2^{(b)}}\opl {\dot A_2^{(r)}}$, so suppose otherwise.
Without loss of generality, every layer is almost red.

First, assume that any two blue vertices in $\QQ$ are comparable, i.e., the blue vertices form a chain. 
Let $b\in[N]$ be a ground element contained in every blue vertex, except for possibly $\varnothing$.
Let $a\in[N]$ be a ground element contained in none of the blue vertices, except for possibly $[N]$. 
Note that the sublattice $\QQ\big|_{\{a\}}^{[N]\setminus\{b\}}$ contains no blue vertex. Since its dimension is $N-2=n$, the sublattice is a red copy of $Q_n$, as desired.

From now on, suppose there are two blue incomparable vertices. Pick two blue vertices $X,Y\in\QQ$ such that $X$ and $Y$ are incomparable, $|X|\le |Y|$, and 
$|Y|-|X|$ is minimal among such pairs, i.e., there are no two blue incomparable $X',Y'\in\QQ$ such that $|X'|\le |Y'|$, and $|Y'|-|X'|<|Y|-|X|$.

\indent Because layers $|X|$ and $|Y|$ are almost red, we see that $1\le |X|<|Y|\le N-1$.
We distinguish three cases, depending on whether $|X|=1$ and $|Y|=N-1$.
\\

\noindent \textbf{Case 1:} $|X|\ge 2$.\smallskip\\
Since $X\not\subseteq Y$, there exists a ground element $a\in X\setminus Y$.
Let 
$$\cF=\{Z\in\QQ:~ |Z|=|X|,\ a\in Z\},$$
so $X\in\cF$. Note that $\cF$ is a layer of the $(N-1)$-dimensional sublattice $\QQ\big|_{\{a\}}^{[N]}$,
therefore the size of $\cF$ is 
$$|\cF|=\binom{N-1}{|X|-1}\ge \binom{N-1}{1}=N-1$$
In particular, there exist two distinct vertices $U_1,U_2\in\cF\setminus\{X\}$.
We claim that $X$, $Y$, $U_1$, and $U_2$ form a copy of $\dot A_2^{(b)}\opl \dot A_2^{(r)}$.
Indeed, $X$ and $Y$ are blue and, since layer $|X|$ is almost red, $U_1$ and $U_2$ are red.
Recall that $\cF$ is a layer of a sublattice and thus an antichain, so $U_1$, $U_2$, and $X$ are pairwise incomparable.
Furthermore, $Y$ is incomparable to each of $U_1$, $U_2$, and $X$, because on the one hand $|U_1|=|U_2|=|X|<|Y|$, and on the other hand $a$ is contained in each of $U_1$, $U_2$, and $X$, but $a\notin Y$.
\\

\noindent \textbf{Case 2:} $|Y|\le N-2$.\smallskip\\
We proceed similarly to Case 1, so we only sketch the proof.
Let $a\in X\setminus Y$, and let $\cF=\{Z\in\QQ:~ |Z|=|Y|,\ a\notin Z\}$.
Observe that $|\cF|\ge N-1$, so we find vertices $U_1,U_2\in\cF$ such that $X$, $Y$, $U_1$, and $U_2$ form a copy of $\dot A_2^{(b)}\opl \dot A_2^{(r)}$.
\\

\noindent \textbf{Case 3:} $|X|=1$ and $|Y|=N-1$.\smallskip\\
%Then $|Y|=N-1$ and $|X|=1$. %, so $|Y|-|X|= N-2$.
Since $X$ and $Y$ are incomparable, there is a ground element $a\in[N]$ such that $X=\{a\}$ and $Y=[N]\setminus \{a\}$.  
Fix some distinct ground elements $b,c\in [N]\setminus \{a\}$.
Assume that there is a blue vertex $U$ in the sublattice $\QQ\big|_{\{b\}}^{[N]\setminus\{c\}}$.
We shall find a contradiction to the minimality of $X$ and $Y$.
Since layer $1$ of the Boolean lattice $\QQ$ is almost red and $X$ is blue, the vertex $\{b\}$ is red, so $|U|\ge 2$. Similarly, $[N]\setminus\{c\}$ is red, which implies that $|U|\le N-2$.

\begin{itemize}
\item If $a\in U$, then $U$ and $Y=[N]\setminus\{a\}$ are incomparable, and $|Y|-|U|<N-2=|Y|-|X|$, contradicting the minimality of $|Y|-|X|$.

\item However, if $a\notin U$, then $U$ and $X=\{a\}$ are incomparable, and $|U|-|X|<|Y|-|X|$, which also contradicts the minimality of $|Y|-|X|$.
\end{itemize}
Therefore, the sublattice $\QQ\big|_{\{b\}}^{[N]\setminus\{c\}}$ is a red copy of $Q_n$.
\end{proof}

\begin{lemma}\label{lem:EHantichain3}
Let $\dot A$ be a colored antichain such that there are three vertices of the same color. Then for sufficiently large $n$,\:\:$\widetilde{R}(\dot A,Q_n)=n+3$.
\end{lemma}
\begin{proof}
The bound $\widetilde{R}(\dot A,Q_n)\ge R(A_3,Q_n)=n+3$ is a consequence of Theorem \ref{thm:EHgenNON} and \cite{QnPA}.
In the remainder of the proof, we bound $\widetilde{R}(\dot A,Q_n)$ from above.
Let $s$ be the number of vertices of $\dot A$ colored in the majority color, so $s\ge 3$. Let $t=s+2^{2s}$.
Let $N=n+3$, and fix an arbitrary blue/red coloring of the Boolean lattice $\QQ=\QQ([N])$ which contains no monochromatic copy of $Q_n$.
We show that there is a copy of $\dot A_s^{(r)} \opl \dot A_s^{(b)}$ in this coloring, so in particular, there is a copy of $\dot A$.
It was shown in \cite{QnPA} that for sufficiently large $n$, $$R(A_t,Q_n)=n+3=N.$$
Since there is neither a blue nor a red copy of $Q_n$, there exists a  red copy $\cA'$ of $A_t$  as well as a   blue copy $\cB'$ of $A_t$ in our coloring.
Note that neither $\varnothing$ nor $[N]$ are contained in the antichains $\cA'$ or $\cB'$, since each of $\varnothing$ and $[N]$ is comparable to every vertex of $\QQ$.

Our proof idea is to find $s$ red vertices in $\cA'$ and $s$ blue vertices in $\cB'$, denoted by $Z_i$, $i\in[2s]$, 
which are ``easily separable'', i.e., such that there exist ground elements $a_i\in Z_i$ and $x_i\notin Z_i$ with $a_i\neq x_j$ for any indices $i,j\in[2s]$.
While we cannot guarantee that the vertices $Z_i$, $i\in[2s]$, form a colored copy of the desired antichain, we shall show that there is a large sublattice $\QQ'$ parallel to the vertices $Z_i$, $i\in[2s]$.
Any antichain of size $2s-1$ in $\QQ'$ contains $s$ monochromatic vertices. 
These monochromatic vertices, together with all $Z_i$'s of the complementary color, shall form a copy of $\dot A_s^{(r)} \opl \dot A_s^{(b)}$, as desired.
\\

%Now, we iteratively define vertices $Z_1,\dots, Z_s$, and ground elements $a_i\in Z_i$ and $x_i\in [N]\setminus Z_i$ for each $i\in[s]$.
Fix a vertex $Z_1\in\cA'$, and let $a_1\in Z_1$ and $x_1 \in [N]\setminus Z_1$ be chosen arbitrarily. We proceed iteratively.
For $i\in\{2,\dots,s\}$, assume that we selected distinct vertices $Z_1,\dots,Z_{i-1}\in\cA'$ and ground elements $a_1,\dots,a_{i-1}, x_1,\dots,x_{i-1}$ such that $a_j\in Z_j$, $x_j\in [N]\setminus Z_j$, and $a_{j}\neq x_{j'}$ for any $j, j'\in[i-1]$.
In the next iterative step, pick a vertex $Z_i\in\cA'$ such that 
\begin{itemize}
\item $Z_i$ is distinct from $Z_1,\dots, Z_{i-1}$,
\item there is an $a_i\in Z_i$ with $a_i\notin \{x_1,\dots,x_{i-1}\}$, and
\item there is an $x_i\in [N]\setminus Z_i$ with $x_i\notin \{a_1,\dots,a_{i-1}\}$.
\end{itemize}

%Note that we allow $a_i=a_j$ and $x_i=x_j$ for $j<i$.
To show that $Z_i$ is well-defined, let $\cF_i$ be the set of vertices that fail at least one of these criteria. 
We need to verify that $|\cF_i|<|\cA'|$.
The vertices in $\cF_i$ are $Z_1,\dots,Z_{i-1}$ as well as all subsets of $\{x_1,\dots,x_{i-1}\}$ 
and all vertices of the form $[N]\setminus X$, where $X\subseteq\{a_1,\dots,a_{i-1}\}$. Thus, the size of $\cF_i$ is 
$$|\cF_i|\le (i-1)+2^{i-1}+2^{i-1}\le (s-1)+2^{s}<t=|\cA'|,$$ 
so a triple $(Z_i, a_i, x_i)$ with the desired properties exists in every step $i$.
After iteration step $i=s$, let $\cA=\{Z_1,\dots,Z_{s}\}$. This subposet of $\cA'$ is a red antichain.

We proceed similarly for $\cB'$, i.e., for $i\in[s]$, we select $Z_{s+i}$, $a_{s+i}$, and $x_{s+i}$.
Pick a vertex $Z_{s+1}\in\cB'$ such that there are $a_{s+1}\in Z_{s+1}$ with $a_{s+1}\notin\{x_1,\dots,x_{s}\}$ and $x_{s+1}\in [N]\setminus Z_{s+1}$ with $x_{s+1}\notin \{a_1,\dots,a_{s}\}$.
This is possible because the number of ``bad'' vertices is $2^{s}+2^{s}<|\cB'|$. Iteratively, let $i\in\{2,\dots,s\}$.
Assume that we defined distinct vertices $Z_{s+1},\dots,Z_{s+i-1}\in\cB'$ and $a_{s+1},\dots,a_{s+i-1}, x_{s+1},\dots,x_{s+i-1}$ such that $a_{j}\in Z_j$, $x_j\in [N]\setminus Z_j$ for $j\in\{s+1,\dots,s+i-1\}$, and $a_{j_1}\neq x_{j_2}$ for any $j_1,j_2\in[s+i-1]$.
We choose $Z_{s+i}\in\cB'$ such that 
\begin{itemize}
\item $Z_{s+i}$ is distinct from $Z_{s+1},\dots, Z_{s+i-1}$,
\item there is an $a_{s+i}\in Z_{s+i}$ such that $a_{s+i}\notin \{x_1,\dots,x_{s+i-1}\}$, and
\item there is an $x_{s+i}\in [N]\setminus Z_{s+i}$ with $x_{s+i}\notin  \{a_1,\dots,a_{s+i-1}\}$.
\end{itemize}

\indent The number of vertices for which one of these properties fails is at most $$(i-1)+2^{s+i-1}+2^{s+i-1}\le(s-1)+2^{2s-1}+2^{2s-1}<t=|\cB'|,$$ 
so $Z_{s+i}$, $a_{s+i}$, and $x_{s+i}$ can be chosen in every step.
Let $\cB=\{Z_{s+1}\dots,Z_{2s}\}$, and note that this is a blue antichain. 
We remark that $\cA$ and $\cB$ are disjoint, because $\cA$ is red and $\cB$ is blue. However, $\cA\cup\cB$ might contain comparable vertices.

Consider the sublattice
$\QQ'=\big\{X\in\QQ: ~ \{x_i :~ i\in [2s]\}\subseteq X \subseteq [N]\setminus \{a_i :~ i\in [2s]\}\big\}$.
This subposet is well-defined, because $a_i\neq x_j$ for any $i,j\in[2s]$.
We claim that $\QQ'$ is parallel to $\cA\cup \cB$.
Let $X\in\QQ'$ and $i\in[2s]$.
Since $x_i\in X\setminus Z_i$ and $a_i\in Z_i\setminus X$, we see that $X$ and $Z_i$ are incomparable, so $\QQ'$ is parallel to $\cA$ and $\cB$.
The dimension of $\QQ'$ is at least $n-4s$. 
For sufficiently large $n$, there exists an antichain $\cP'$ on $2s-1$ vertices in~$\QQ'$.
In particular, $\cP'$ contains a monochromatic antichain $\cP$ on $s$ vertices.
If $\cP$ is blue, then $\cA\cup\cP$ is a copy of $\dot A_s^{(r)} \opl \dot A_s^{(b)}$.
If $\cP$ is red, then $\cP\cup\cB$ is a copy of $\dot A_s^{(r)} \opl \dot A_s^{(b)}$.
\end{proof}

\begin{proof}[Proof of Theorem \ref{thm:EHantichain}]
Lemma \ref{lem:EHantichain1} implies that $\widetilde{R}({\dot A_1^{(b)}}\opl {\dot A_1^{(r)}},Q_n)=n+2$.
By Lemma \ref{lem:EHantichain2}, $\widetilde{R}({\dot A_2^{(b)}}\opl {\dot A_2^{(r)}},Q_n)=n+2$,
thus also
$$n+2= \widetilde{R}({\dot A_1^{(b)}}\opl {\dot A_1^{(r)}},Q_n) \le \widetilde{R}({\dot A_2^{(b)}}\opl {\dot A_1^{(r)}},Q_n) \le \widetilde{R}({\dot A_2^{(b)}}\opl {\dot A_2^{(r)}},Q_n)=n+2,$$
and similarly $\widetilde{R}({\dot A_1^{(b)}}\opl {\dot A_2^{(r)}},Q_n)=n+2$.
For any other non-monochromatically colored antichain, the poset Erd\H{o}s-Hajnal number is determined by Lemma \ref{lem:EHantichain3}.
\end{proof}

%%%%%%%%%%%%%%%%%%%%%%%%%%%%%%%%%%%%%%%%%%%%%%%%%%%%%%%%%%%%%%%%%%%%%%%
%%%
%%%%%%%%%%%%%%%%%%%%%%%%%%%%%%%%%%%%%%%%%%%%%%%%%%%%%%%%%%%%%%%%%%%%%%%

\section{Forbidden chains}\label{sec:EHchain}

\subsection{Proof of Theorem \ref{thm:EHreduction}}
%\subsection{Phase of a vertex}\label{sec:EHchain_not}
%The following notation is used throughout this section.
Throughout this subsection, let $\dot  C$ be a fixed colored chain on $t$ vertices $Z_1< Z_2 < \dots < Z_t$. 
For $i\in[t]$, we denote by~$\dot C\big|^{Z_{i}}_{Z_{1}}$ the subposet of $C$ consisting of its $i$ smallest vertices $Z_1<\dots<Z_i$, colored as in~$\dot C$. 
Additionally, let $\dot C\big|^{Z_{0}}_{Z_{1}}$ be the empty colored poset.
In this subsection, $\QQ$ is a Boolean lattice with a fixed $\dot C$-free blue/red coloring.
We partition the vertices of $\QQ$ into so-called \textit{phases}.
The \textit{$i$-th phase} of $\QQ$ with respect to $\dot C$ is defined as the family of vertices
$$\cF^{\dot C}_{i}=\left\{X\in\QQ : ~ \QQ\big|^X_\varnothing\text{ contains a copy of }\dot C\big|^{Z_{i-1}}_{Z_{1}}
\text{, but no copy of }\dot C\big|^{Z_{i}}_{Z_{1}}\right\}.$$
Here, $\QQ\big|^X_\varnothing$ inherits the coloring from $\QQ$. See Figure \ref{fig:QnEH:phase} for an example of phases of $Q_4$.
We remark that $\cF^{\dot C}_{i}$ might be empty. 

\begin{figure}[h]
\centering
\includegraphics[scale=0.62]{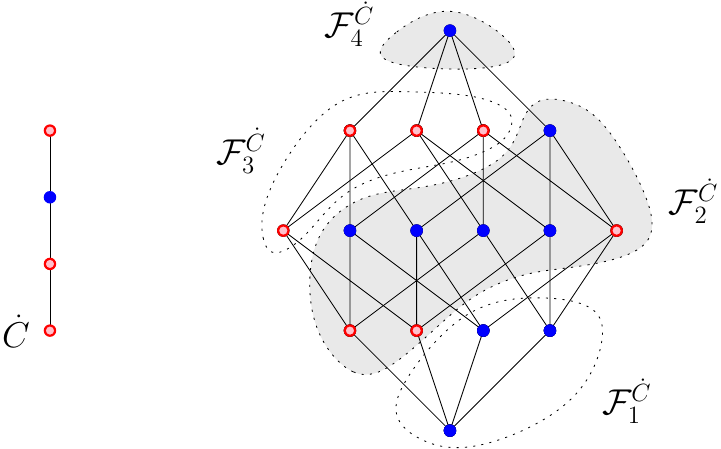}
\caption{A colored chain $\dot C$ and a $\dot C$-free blue/red coloring of $Q_4$ with sets $\cF^{\dot C}_{i}$, $i\in[4]$.}
\label{fig:QnEH:phase}
\end{figure}

Denote the color of $Z_i$, the $i$-th vertex of $\dot C$, by $c_i\in\{\text{blue}, \text{red}\}$, and let $\bar{c}_i$ be its complementary color.
Let $I(\dot C)$ be the set of indices for which there is no \textit{color switch} in $\dot C$, i.e.,
$$I(\dot C)=\big\{i\in\{2,\dots,t\} : ~ c_i= c_{i-1}\big\}.$$
In our example, $I(\dot C)=\{2\}$.
For $i\in [t]$, we define $\cA_i$ as the set of minimal vertices of~$\cF^{\dot C}_{i}$.
For example, in Figure \ref{fig:QnEH:phase}, the set $\cA_2$ consists of the three red vertices in $\cF^{\dot C}_{2}$.

The following properties are immediate, so we omit the proof.
\begin{lemma}\label{lem:QnEH:phase_basic} \ 
\begin{itemize}
\item[(i)] The families $\cF^{\dot C}_{1},\dots,\cF^{\dot C}_{t}$ partition $\QQ$.
\item[(ii)] Let $X,Y\in\QQ$ with $X\in \cF^{\dot C}_{i}$ and $Y\in \cF^{\dot C}_{j}$ for some $i,j\in[t]$. If $X \subseteq Y$, then $i\le j$. %In particular, if $i>j$ then $X \not \subseteq Y$.
\end{itemize}
\end{lemma}

The next lemma shows that the color of each vertex in $\QQ$ is determined by its phase.
\begin{lemma}\label{lem:phase_color} \ 
%Let $\dot C$ be a colored chain on vertices $Z_1< Z_2 < \dots < Z_{t}$. 
%Denote the color of $Z_i$ by $c_i\in\{\text{blue}, \text{red}\}$, and let $\bar{c}_i$ be its complementary color.
%Fix a $\dot C$-free blue/red coloring of a Boolean lattice $\QQ$. Then:
\begin{enumerate}
\item[(i)] Every vertex in $\cF^{\dot C}_{1}$ has color $\bar{c}_1$.
\item[(ii)] Let $2\le i \le t$ with $c_i\neq c_{i-1}$. Then every vertex in $\cF^{\dot C}_{i}$ has color $\bar{c}_i$.
\item[(iii)] Let $2\le i \le t$ with $c_i= c_{i-1}$. Then every vertex of $\cA_{i}$  has color $c_i$, and every vertex in $\cF^{\dot C}_{i}\setminus \cA_i$ has the complementary color $\bar{c}_i$.
\end{enumerate}
\end{lemma}

\begin{proof}
Part (i) is immediate from the definition of $\cF^{\dot C}_{1}$.

For part (ii), consider an index $i\ge 2$ with $c_i\neq c_{i-1}$. Let $X$ be an arbitrary vertex in~$\cF^{\dot C}_{i}$. 
By definition of $\cF^{\dot C}_{i}$, there is a copy $\dot\cD$ of $\dot C\big|^{Z_{i-1}}_{Z_{1}}$ in $\QQ\big|^X_\varnothing$.
If $X$ has color $c_i=\bar{c}_{i-1}$, then $X$ has a different color than the maximal vertex of $\dot\cD$ and is larger than any vertex of $\dot\cD$, thus $X\notin \dot\cD$.
In particular, by adding the vertex $X$ to the colored chain $\dot\cD$, we obtain a copy of $\dot C\big|^{Z_{i}}_{Z_{1}}$ in $\QQ\big|^X_\varnothing$.
This is a contradiction to the assumption $X\in \cF^{\dot C}_{i}$.
Thus, the color of $X$ is $\bar{c}_i$.

For part (iii), let $i\ge 2$ with $c_i= c_{i-1}$, i.e., $i\in I(\dot C)$, and fix a vertex $X\in \cF^{\dot C}_{i}$.
\begin{itemize}
\item If $X\in \cA_i$, then $X$ is minimal with the property that $\QQ\big|^X_\varnothing$ contains a copy of $\dot C\big|^{Z_{i-1}}_{Z_{1}}$. 
In particular, $X$ is contained in a copy $\dot\cD$ of $\dot C\big|^{Z_{i-1}}_{Z_{1}}$ in $\QQ\big|^X_\varnothing$. 
The vertex $X$ is the maximal vertex of $\QQ\big|^X_\varnothing$, thus $X$ is also the maximal vertex of $\dot\cD$.
In particular, $X$ has color $c_{i-1}=c_i$.

\item If $X\notin\cA_i$, then there is a vertex $A\in \cF^{\dot C}_{i}$ such that $A\subset X$. 
Let $\dot \cD$ be a copy of $\dot C\big|^{Z_{i-1}}_{Z_{1}}$ in $\QQ\big|^A_\varnothing$.  
If $X$ has color $c_i$, then $\dot \cD$ and $X$ form a copy of $\dot C\big|^{Z_{i}}_{Z_{1}}$ in $\QQ\big|^X_\varnothing$, contradicting that $X$ is a vertex of $\cF^{\dot C}_{i}$. 
Therefore, $X$ has color $\bar{c}_i$.
\end{itemize}
\vspace*{-2em}
\end{proof}

%%%%%%%%%%%%%%%%%%%%%%%%%%%%%%%

%\subsection{Proof of Theorem \ref{thm:EHreduction}}

\begin{proof}[Proof of Theorem \ref{thm:EHreduction}] 
Let $\dot  C$ be a colored chain on vertices $Z_1<\dots<Z_t$.
Recall that $\lambda=\lambda(\dot  C)$ is the maximal integer $\ell$ such that $\dot  C$ contains a copy of $\dot C_\ell^{(rbr)}$ or $\dot C_\ell^{(brb)}$. 
%We shall bound $\widetilde{R}(\dot  C,Q_n)$ in terms of $\widetilde{R}(\dot C^{(rbr)}_{\lambda},Q_n)$. 
By switching the colors, we can suppose without loss of generality that the minimal vertex $Z_1$ of $\dot  C$ is red.
Observe that there is a largest alternating chain in $\dot C$ which contains $Z_1$.
In particular, there exists a largest alternating chain in $\dot C$ that is red-alternating, i.e., $\dot  C$ contains a copy of $\dot C_\lambda^{(rbr)}$.

For the lower bound on $\widetilde{R}(\dot  C,Q_n)$, note that any $\dot C^{(rbr)}_{\lambda}$-free colored Boolean lattice is also $\dot C$-free, so 
$\widetilde{R}(\dot  C,Q_n)\ge \widetilde{R}(\dot C^{(rbr)}_{\lambda},Q_n)$.

To show the upper bound on $\widetilde{R}(\dot  C,Q_n)$, we present a non-constructive lower bound on $\widetilde{R}(\dot C^{(rbr)}_{\lambda},Q_n)$, 
in terms of $\widetilde{R}(\dot  C,Q_n)$.
Let $N=\widetilde{R}(\dot C,Q_n)-1$ and $\QQ=\QQ([N])$. Select an arbitrary blue/red coloring of $\QQ$ which is $\dot C$-free and contains no monochromatic copy of $Q_n$. This coloring exists because $N<\widetilde{R}(\dot C,Q_n)$.
In $\QQ$, we shall find a copy $\QQ'$ of a Boolean lattice of dimension $N-t+\lambda$ which is colored $\dot C^{(rbr)}_{\lambda}$-free.
This proves that $\widetilde{R}(\dot C^{(rbr)}_{\lambda},Q_n)>N-t+\lambda$, implying the desired bound 
$\widetilde{R}(\dot  C,Q_n)=N+1\le \widetilde{R}(\dot C^{(rbr)}_{\lambda},Q_n)+t-\lambda$.

Next, we construct $\QQ'\subseteq \QQ$.
For $i\in[t]$, we denote by $\cF_i=\cF^{\dot C}_i$ the $i$-th phase of $\QQ$ with respect to $\dot C$.
Let $I=I(\dot C)$, i.e., the set of indices for which there is no color switch in $\dot C$. Observe that $|I|=t-\lambda$.
Recall that $\cA_i$ denotes the set of minimal vertices in $\cF_{i}$. Note that each $\cA_i$ is an antichain.
Given any $m$ antichains in $\QQ([N])$ for some $m\in\N$, consider the auxiliary coloring in which the antichains are blue and all other vertices are red. 
Then Corollary \ref{cor:chain} implies that $\QQ([N])$ contains a copy of an $(N-m)$-dimensional Boolean lattice not containing a single vertex of any of the antichains.
Thus, there is a copy~$\QQ'$ of a Boolean lattice of dimension $N-|I|=N-t+\lambda$ such that $\QQ'$ is disjoint from every $\cA_i$, $i\in I$.

For every $i\in[t]$, let $\cF'_i=\cF_i \cap \QQ'$, see Figure \ref{fig:QnEH_phaseF}. 
By Lemma \ref{lem:phase_color}, each $\cF'_i$, $i\in[t]$, is monochromatically colored with color $\bar{c}_i$.
%In particular, the number of color switches of the $\cF'_i$'s, i.e., indices $i\ge 2$ such that consecutive $\cF_{i-1}$ and $\cF_{i}$ have distinct colors, is equal the number of color switches in $\dot C$, which is $\lambda -1$.
Furthermore, by Lemma \ref{lem:QnEH:phase_basic} (i), we see that $\cF'_1,\dots,\cF'_t$ partition $\QQ'$.

\begin{figure}[h]
\centering
\includegraphics[scale=0.62]{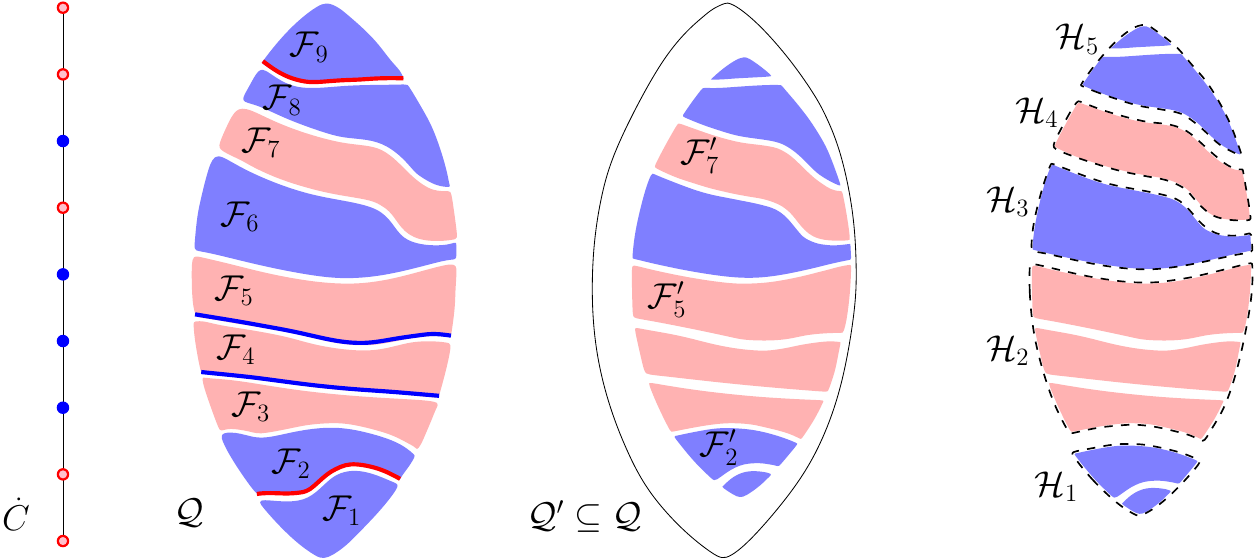}
\caption{A colored chain $\dot C$, families $\cF_i$ in $\QQ$, $\cF'_i$ in $\QQ'$, and $\cH_j$ partitioning $\QQ'$, where $t=9$, $s=5$, and $\lambda=5$.}
\label{fig:QnEH_phaseF}
\end{figure}

Next, we define vertex families $\cH_1,\dots,\cH_s$ partitioning $\QQ'$, by merging families $\cF'_i$, $i\in[t]$.
That is, let each $\cH_j$ be the union of consecutive phases $\cF'_i$'s of the same color, such that for $j\ge 2$, $\cH_{j}$ and $\cH_{j-1}$ have different colors, 
and such that consecutive $H_j$'s contain consecutive phases. An illustration of this merging is given in Figure \ref{fig:QnEH_phaseF}. 
Observe that the number of color switches of $\cH_j$'s, i.e., indices $j\ge 2$ for which $\cH_{j}$ and $\cH_{j-1}$ have distinct colors, 
is equal to the number of color switches of $\cF'_i$'s.
Recalling that each $\cF'_i$ has color $\bar{c}_i$, this quantity is equal to the number of color switches in $\dot C$, which is $\lambda-1$.
Therefore, $s\le \lambda$.

%Recalling that $\lambda$ is the length of the largest alternating chain in $\dot C$ and that the color of $\cF'_i$ is $\bar{c}_i$, we know that the number of color switches of the $\cF'_i$'s, i.e., indices $i$ such that $\cF_i$ and $\cF_{i-1}$ have distinct colors is at most $\lambda$. 
Since the families $\cH_j$, $j\in[s]$, consist of consecutive phases and by Lemma~\ref{lem:QnEH:phase_basic}~(ii), we have that for any $X\in\cH_{j_1}$ and $Y\in\cH_{j_2}$,
\begin{equation}
\text{ if }\quad X\subseteq Y, \quad \text{ then }\quad j_1\le j_2.
\label{eq:QnEH:chunk}
\end{equation}

To show that $\QQ'$ is $\dot C^{(rbr)}_{\lambda}$-free, we assume that there is a red-alternating chain $\cU$ of length $\lambda$ in $\QQ'$, 
say on vertices $U_1\subset \dots \subset U_\lambda$.
\begin{itemize}
\item If there is an $\cH_j$ which contains two vertices of $\cU$, say $U_{\ell}$ and $U_{\ell'}$ for some $\ell,\ell'\in[\lambda]$ with $\ell<\ell'$,
then (\ref{eq:QnEH:chunk}) implies that $U_{\ell +1}\in\cH_j$. Note that $U_{\ell}$ and $U_{\ell+1}$ have distinct colors. We arrive at a contradiction, because $\cH_j$ is monochromatic.

\item If every $\cH_j$, $j\in[s]$, contains at most one vertex of $\cU$, then every $\cH_j$ contains exactly one vertex of $\cU$, since $\cU$ has length $\lambda\ge s$.
In particular, $\cH_1\cap \cU$ is not empty. By (\ref{eq:QnEH:chunk}), $U_1\in \cH_1$.
The chain $\cU$ is red-alternating, so $U_1$ is red. However, $\cH_1$ has the color of $\cF'_1$, i.e., $\bar{c}_1$.
Recalling that $Z_1$, the minimal vertex of $\dot C$, is red, we conclude that $\cH_1$ is blue. This is a contradiction.
\end{itemize}
\vspace*{-2em}
\end{proof}

%%%%%%%%%%%%%%%%%%%%%%%%%%%%%%%

\subsection{Proof of Theorem \ref{thm:EHchain}}

We break down the proof of Theorem \ref{thm:EHchain} into three parts:
Theorem \ref{thm:EHchain} is immediate from Lemmas \ref{lem:EHchain1}, \ref{lem:EHchain2}, and \ref{lem:EHchain3}.

\begin{lemma}\label{lem:EHchain1}
For every $n\in\N$,\:\:$\widetilde{R}(\dot C^{(rbr)}_{2},Q_n)=\widetilde{R}(\dot C^{(rbr)}_{3},Q_n)=2n$. 
\end{lemma}

\begin{proof}
The lower bound is a consequence of Theorem \ref{thm:EHgenFUL}. 
Since $\widetilde{R}(\dot C^{(rbr)}_{2},Q_n)\le \widetilde{R}(\dot C^{(rbr)}_{3},Q_n)$, 
it remains to show that $\widetilde{R}(\dot C^{(rbr)}_{3},Q_n)\le2n$.
Let $\QQ=\QQ([2n])$, and pick an arbitrary blue/red coloring of $\QQ$. We shall find a copy of $\dot C^{(rbr)}_{3}$ or a monochromatic copy of $Q_n$ in this coloring. 
If the longest red chain in $\QQ$ has length at most $n$, Corollary~\ref{cor:chain} guarantees the existence of a blue copy of a Boolean lattice with dimension at least $n$.
So, suppose that there exists a red chain of length $n+1$. 
We denote its minimal element by $A$ and its maximal element by $B$, i.e., $A\subseteq B$ and $|B|-|A|\ge n$.
If there is a blue vertex $Z$ in the sublattice $\QQ\big|_A^B$, then the vertices $A$, $Z$, and $B$ form a copy of $\dot C^{(rbr)}_{3}$.
Otherwise, $\QQ\big|_A^B$ is a red copy of a Boolean lattice of dimension $|B|-|A| \ge n$.
\end{proof}

\begin{lemma}\label{lem:EHchain2}
Let $n\in\N$ and $t\ge 3$. Then $\widetilde{R}(\dot C^{(rbr)}_{t},Q_n)\le (t-1)n$. 
\end{lemma}
\begin{proof}
We prove this statement using induction. The base case $t=3$ is shown in Lemma~\ref{lem:EHchain1}.
Suppose that $\widetilde{R}(\dot C^{(rbr)}_{t},Q_n)\le(t-1)n$ for some $t\ge3$.
We shall show that $\widetilde{R}(\dot C^{(rbr)}_{t+1},Q_n)\le tn.$
Let $N=tn$ and choose an arbitrary blue/red coloring of the host Boolean lattice $\QQ=\QQ([N])$.
Fix any vertex $Z\in\QQ([N])$ with $|Z|=N-n=(t-1)n$, and consider the sublattices $\QQ\big|^Z_\varnothing$ and $\QQ\big|_Z^{[N]}$. %, see Figure \ref{fig:QnEH_induction}.
By induction, we find in $\QQ\big|^Z_\varnothing$ either a monochromatic copy of $Q_n$, which completes the proof, or a copy $\dot \cD$ of $\dot C^{(rbr)}_{t}$.
In the latter case, let $X\in \QQ\big|_Z^{[N]}$ be a vertex colored differently than the maximal vertex in $\dot \cD$.
Then $\dot \cD$ and $X$ form a copy of $\dot C^{(rbr)}_{t+1}$. 
If there exists no such vertex $X$, then the sublattice $\QQ\big|_Z^{[N]}$ is a monochromatic copy of $Q_n$.
\end{proof}
%
%\begin{figure}[h]
%\centering
%\includegraphics[scale=0.62]{figs/QnEH_induction}
%\caption{Sublattices $\QQ\big|^Z_\varnothing$ and $\QQ\big|_Z^{[N]}$ in Lemma \ref{lem:EHchain2}.}
%\label{fig:QnEH_induction}
%\end{figure}

\begin{lemma}\label{lem:EHchain3}
For sufficiently large $n$,\:\:$\widetilde{R}(\dot C^{(rbr)}_{4},Q_n)> 2.02n$.
\end{lemma}

\noindent\textbf{Outline of the proof idea  for Lemma \ref{lem:EHchain3}:} Let $ c  =0.02$. Let $n$ be a natural number, and let $N=(2+ c )n$.
First, in Lemma \ref{lem:EHchain_main}, we use a probabilistic argument to find two families $\cS$ and $\cT$ of vertices in the Boolean lattice $\QQ([N])$ in layers $(1- c )n$ and $(1+2 c )n$, respectively, which have two properties:
\begin{enumerate}
\item[(1)] every vertex in $\cS$ is incomparable to every vertex in $\cT$, and 
\item[(2)] both $\cS$ and $\cT$ are ``dense'' in their respective layer.
\end{enumerate}

Afterwards, we formally define a blue/red coloring in Construction \ref{constr:EHchain}, as illustrated in Figure~\ref{fig:EHcoloring}. 
We need (1) to ensure that this construction is well-defined.
As a final step, we shall show that there is no monochromatic copy of $Q_n$ and no copy of  $\dot C^{(rbr)}_{4}$  in our construction, for which we use (2).
Recall that we omit floors and ceilings where appropriate.

\begin{figure}[h]
\centering
\includegraphics[scale=0.62]{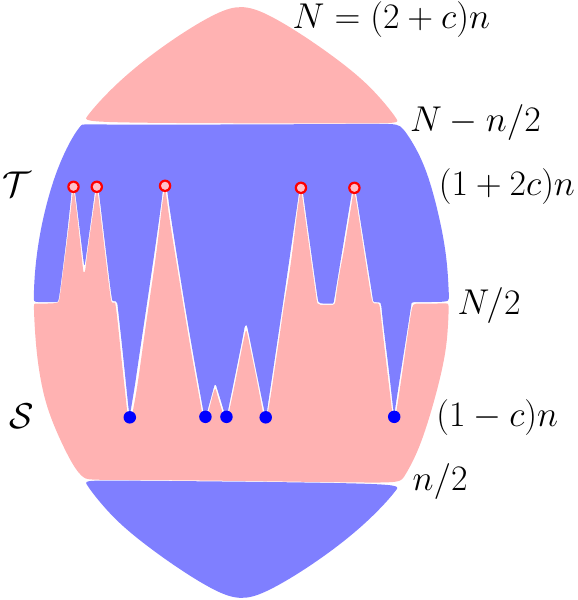}
\caption{Blue/red coloring of $\QQ([N])$ based on $\cS$ and $\cT$ in Construction \ref{constr:EHchain}.}
\label{fig:EHcoloring}
\end{figure}

\begin{lemma}\label{lem:EHchain_main}
Let $ c =0.02$. Let $N=(2+ c )n$ for sufficiently large $n$.  
Then there exist families $\cS$ and $\cT$ of vertices in $\QQ([N])$ with the following properties:
\begin{enumerate}
\item[(i)] For every $S\in\cS$,\:\:$|S|=(1- c )n$. For every $T\in\cT$,\:\:$|T|=(1+2 c )n$.
\item[(ii)] Every two vertices $S\in\cS$ and $T\in\cT$ are incomparable.
\item[(iii)] For every pair of disjoint sets $A,B\subseteq [N]$ with $|A|=\tfrac{n}{2}$ and $|B|=n$, there exists an $S\in\cS$ with $S \subseteq A\cup B$ and $|B\cap S|\le \tfrac{n}{2}$.
\item[(iv)] For every pair of disjoint sets $A,B\subseteq [N]$ with $|A|=\tfrac{n}{2}$ and $|B|=n$, there exists a $T\in\cT$ with $T\supseteq [N]\setminus (A\cup B)$ and $|B\setminus T|\le  \tfrac{n}{2}$.
\end{enumerate}
\end{lemma}

\begin{proof}
First, we introduce several families of vertices in $\QQ([N])$.
Let $s=(1- c )n$ and $t=(1+2 c )n$, and denote the corresponding layers of $\QQ([N])$ by 
$$\cL_s=\big\{Z\in\QQ([N]):~|Z|=s\big\}\quad \text{ and }\quad \cL_t=\big\{Z\in\QQ([N]):~|Z|=t\big\}.$$
Let 
$$\text{Cone}_s=\big\{ \cK_s(A,B) : ~ A,B\subseteq [N],\ A\cap B=\varnothing,\ |A|=\tfrac{n}{2},\ |B|=n\big\}$$
be a collection of \textit{cones} $\cK_s(A,B)$, which are defined as
$$\cK_s(A,B)=\big\{S\in \cL_s : ~ S \subseteq A\cup B,\ |B\cap S|\le \tfrac{n}{2}\big \},$$
as illustrated in Figure \ref{fig:QnEH_cones}. Similarly, let
$$\text{Cone}_t=\big\{ \cK_t(A,B) : ~ A,B\subseteq [N],\ A\cap B=\varnothing,\ |A|=\tfrac{n}{2},\ |B|=n\big\},$$
where a \textit{cone} $\cK_t(A,B)$ is a family of vertices given by
$$\cK_t(A,B)=\big\{T\in \cL_t : ~ T\supseteq [N]\setminus (A\cup B),\ |B\setminus T|\le  \tfrac{n}{2}\big \}.$$
Furthermore, we define the \textit{neighborhood} of a vertex $S\in\cL_s$ as 
$$\cN_t(S)=\big\{T\in \cL_t : ~ T\supseteq S \big\}.$$

\begin{figure}[h]
\centering
\includegraphics[scale=0.62]{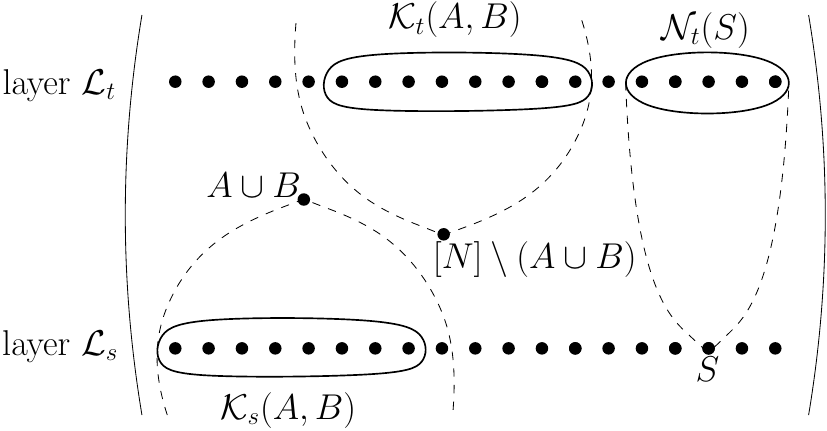}
\caption{Examples for families $\cK_s(A,B)$, $\cK_t(A,B)$, and $\cN_t(S)$.}
\label{fig:QnEH_cones}
\end{figure}

We shall find families $\cS$ and $\cT$ such that
\begin{enumerate}
\item[(i')] $\cS\subseteq \cL_s$ and $\cT\subseteq \cL_t$,
\item[(ii')] for every $S\in \cS$,\:\:$\cN_t(S)\cap \cT=\varnothing$,
\item[(iii')] for every $\cK\in\text{Cone}_s$, there exists an $S\in\cK\cap \cS$, and
\item[(iv')] for every $\cK\in\text{Cone}_t$, there is a $T\in\cK\cap\cT$.
\end{enumerate}
Each property (i') to (iv') implies the respective property (i) to (iv).
Summarizing these properties, the subposet $\cS\cup\cT$ can be described as an antichain that is a ``transversal'' of $\text{Cone}_s\cup \text{Cone}_t$.

To find the desired $\cS$ and $\cT$, we consider the following two random families. Let $p=0.77^n$.
Randomly draw a family $\cS'$ by independently including each $S\in \cL_s$ with probability~$p$.
Similarly, draw a family $\cT$ by including each $T\in \cL_t$ independently with probability~$p$.

We say that an event $E(n)$ holds with \textit{high probability}, abbreviated by \textit{w.h.p.}, if $\PPP(E(n))\to 1$ for $n\to \infty$.
In the following, we shall show that with high probability, $\cS'\cup\cT$ has a \textit{large} intersection with every $\cK\in \text{Cone}_s\cup \text{Cone}_t$, 
i.e., $\cS'\cup\cT$ is a ``strong transversal'' of $\text{Cone}_s\cup \text{Cone}_t$.
Afterwards, we deterministically refine $\cS'$, by deleting vertices which are ``bad'' with respect to property (ii'), resulting in a family $\cS\subseteq\cS'$. 
Lastly, we shall verify that $\cS$ has a \textit{non-empty} intersection with every cone $\cK\in \text{Cone}_s$.
%Let $\cS$ be obtained from $\cS'$ by removing every $S\in\cS'$ for which $\cN_t(S)\cap \cT\neq\varnothing$, i.e., for which there is a $T\in\cT$ such that $S\subseteq T$.
%Note that $\cS$ and $\cT$ possess properties (i') and (ii').

%In the remainder of the proof, we show that with high probability, $\cS$ and $\cT$ have properties (iii') and (iv').
%The proof idea for property (iii') is to show that for every cone $\cK\in\text{Cone}_s$, the size of $\cK\cap\cS'$ is large, and afterwards verify that at least one $S\in\cK\cap\cS'$ is still contained in $\cS$ after the refinement step.

Stirling's formula provides that $N!=\Theta(\sqrt{N})\left(\frac{N}{e}\right)^N$.
Throughout this proof, we repeatedly apply the following consequence of Stirling's formula. 
For positive constants $C> d$,
\begin{eqnarray}
\binom{C n}{d n}&=&\frac{\Theta(1)\sqrt{Cn}}{\sqrt{dn}\sqrt{(C-d)n}} \frac{(Cn)^{Cn}}{e^{Cn}} \frac{e^{dn}}{(dn)^{dn}}\frac{e^{(C-d)n}}{((C-d)n)^{(C-d)n}}\nonumber \\
&=&\Theta\left(\frac{1}{\sqrt{n}}\right) \left(\frac{C^C}{d^d (C-d)^{C-d}}\right)^n.\label{eq:QnEH:stirling}
\end{eqnarray}

\noindent \textbf{Claim 1:} With high probability, every cone $\cK\in\text{Cone}_s$ has an intersection with the (unrefined) family $\cS'$ of size
$|\cK\cap\cS'|\ge 1.66^n$.\smallskip\\
\textit{Proof of Claim 1.}
For arbitrary fixed, disjoint $A,B\subseteq [N]$ with $|A|=\tfrac{n}{2}$ and $|B|=n$, let $\cK=\cK_s(A,B)\in\text{Cone}_s$.
Each element in $\cK$ is included in $\cS'$ independently with probability $p$. Thus,
$$|\cK\cap\cS'|\sim \text{Bin}\big(|\cK|,p\big),\quad \text{ and }\quad \mathbb{E}(|\cK\cap\cS'|)=|\cK|\cdot p= |\cK|\cdot 0.77^n.$$

We shall bound $|\cK|$ from below. 
If $S\in \cK_s(A,B)$, then $S$ consists of $s$ elements, so $|B\cap S|=|S|-|A\cap S|\ge s-|A|\ge (\tfrac12-c)n$.
Thus, $(\tfrac12-c)n\le |B\cap S|\le \tfrac{n}{2}$.
Using (\ref{eq:QnEH:stirling}) and $c=0.02$, we see that the size of $\cK$ is 
\begin{eqnarray*}
|\cK|&=&\sum_{i=0}^{ c  n} \binom{|A|}{s-(n/2-i)}\binom{|B|}{n/2-i}\\
&\ge &\binom{|A|}{ s-n/2 }\binom{|B|}{n/2} \\
&=& \binom{n/2}{ n/2 - c n}\binom{n}{n/2}\\
&=& \Theta\left(\frac{1}{n}\right) \left(\frac{1}{c ^{ c }\ (1/2- c )^{1/2- c }\ (1/2)^{1/2}} \right)^n \\
&\ge & 2.17^n,
\end{eqnarray*}
where the last bound holds for sufficiently large $n$. 
%Similarly, we find an upper bound
%\begin{eqnarray*}
%|\cK|&=&\sum_{i=0}^{ c  n} \binom{|A|}{s-(n/2-i)}\binom{|B|}{n/2-i}\\
%&\le &(cn+1)\binom{|A|}{ s-n/2 }\binom{|B|}{n/2} \\
%&=& \Theta\left(1\right) \left(\frac{1}{c ^{ c }\ (1/2- c )^{1/2- c }\ (1/2)^{1/2}} \right)^n \\
%&\le & 2\left(2.2\right)^n,
%\end{eqnarray*}
%for sufficiently large $n$.
In particular, for large $n$,
$$
\mathbb{E}(|\cK\cap\cS'|)=|\cK|\cdot p \ge 2.17^n \cdot 0.77^n \ge 2 \cdot 1.66^n.
$$
The multiplicative form of Chernoff's inequality, see Corollary 23.7 in Frieze and Karo\'{n}ski \cite{FK}, provides that for a random variable $X$ with binomial distribution and for $0<a<1$,
$$
\PPP\big(X\le (1-a)\EEE(X)\big)\le \exp\left(-\frac{\EEE(X) a^2}{2}\right).
$$
Using this inequality for $X=|\cK\cap\cS'|$ and $a=\tfrac12$, 
\begin{eqnarray*}
\PPP\left(|\cK\cap\cS'|< 1.66^n\right) & \le & \PPP\left(|\cK\cap\cS'|\le \left(1-\frac{1}{2}\right) \mathbb{E}(|\cK\cap\cS'|)\right) \\
%& \le & \PPP\left(|\cK\cap\cS'|\le \frac{\mathbb{E}(|\cK\cap\cS'|)}{2}\right)\\
& \le & \exp\left(-\frac{\mathbb{E}(|\cK\cap\cS'|)}{8}\right)\\
&\le & \exp\left(-4\cdot 1.66^n\right)
\end{eqnarray*}
Let $X_{\cK\cap\cS'}$ be the random variable counting cones $\cK\in\text{Cone}_s$ such that $|\cK\cap\cS'|<1.66^n$.
The expected value of $X_{\cK\cap\cS'}$ is
\begin{eqnarray*}
\EEE(X_{\cK\cap\cS'}) & = & \sum_{\cK\in\text{Cone}_s} \PPP\left (|\cK\cap\cS'|<1.66^n\right)\\
&\le & \sum_{\substack{B\subseteq[N],\\ |B|=n}} \ \  \sum _{\substack{A\subseteq[N]\setminus B,\\ |A|=n/2}}   \exp\left(-4\cdot 1.66^n\right) \\
&\le &2^{2N} \exp\left(-4\cdot1.66^n\right)\\
&\le &2^{4.04n} \exp\left(-4\cdot1.66^n\right) \to 0\text{ for }n\to \infty,
\end{eqnarray*}
thus w.h.p., $X_{\cK\cap\cS'}=0$, i.e., every cone $\cK\in\text{Cone}_s$ has a large intersection with $\cS'$. This proves Claim 1.
\\

\noindent \textbf{Claim 2:} With high probability, $|\cK\cap\cT|\ge 1.66^n$ for every $\cK\in\text{Cone}_t$. In particular, w.h.p., $\cT$ has property (iv'). \smallskip\\
\textit{Proof of Claim 2.}
This claim can be shown similarly to Claim 1, so we only provide a sketch of the proof.
Fix a $\cK=\cK_t(A,B)\in\text{Cone}_t$. Note that 
$$|\cK\cap\cT|\sim \text{Bin}\big(|\cK|,p\big),\quad \text{ and }\quad \mathbb{E}(|\cK\cap\cT|)=|\cK|\cdot p= |\cK|\cdot 0.77^n.$$
The size of $\cK$ is bounded from below as follows:
\begin{eqnarray*}
|\cK|&=&  \sum_{i=0}^{ c  n} \binom{|A|}{t-\big|[N]\setminus (A\cup B)\big|-(n/2+i)}\binom{|B|}{n/2+i}\\
& \ge & \binom{n/2}{ c  n}\binom{n}{n/2} \ge  2.17^n.
\end{eqnarray*}
Thus, $\mathbb{E}(|\cK\cap\cT|)=|\cK|\cdot p \ge 2 \cdot 1.66^n$.
Analogously to Claim 1, this implies that w.h.p., $|\cK\cap\cT|\ge 1.66^n$ for every cone $\cK\in\text{Cone}_t$.
\\

We say that a family of vertices $\cK\subseteq \cL_s$ is \textit{bad} if for every $S\in\cK\cap\cS'$, the intersection $\cN_t(S)\cap \cT$ is non-empty.
We shall show that w.h.p., there exists no bad cone $\cK\in\text{Cone}_s$.
\\

\noindent \textbf{Claim 3:} Let $\cK\in\text{Cone}_s$ such that $|\cK\cap\cS'|\ge 1.66^n$.
Then $\PPP( \cK \text{ is bad})\le 0.98^{n(1.04)^{n}}$. \smallskip\\
\textit{Proof of Claim 3.}
First, we evaluate $\PPP( \cK' \text{ is bad})$ for a subfamily $\cK'\subseteq\cK\cap\cS'$.
We construct~$\cK'$ such that the neighborhoods $\cN_t(S)$, $S\in\cK'$, are pairwise disjoint, by using a greedy process.
Let $\cK^{0}=\cK\cap\cS'$. Pick a vertex $S_1\in\cK^{0}$ to be added to~$\cK'$.
Let $\cK^{1}$ be the set of remaining vertices $S\in\cK^{0}\setminus\{S_1\}$ for which the neighborhood $\cN_t(S)$ is disjoint from $\cN_t(S_1)$.
Iteratively for $i\ge 2$, as long as $\cK^{i-1}\neq \varnothing$, pick a vertex $S_{i}\in\cK^{i-1}$ to be added to $\cK'$.
Let $\cK^{i}\subseteq \cK^{i-1}$ be the set of vertices $S\in\cK^{i-1}\setminus\{S_{i}\}$ for which $\cN_t(S)\cap \cN_t(S_{i})=\varnothing$.

If $\cK^{i-1}=\varnothing$, we stop the process, and let $\cK'=\{S_1,\dots,S_{i-1}\}$. 
By construction, the families $\cN_t(S)$, $S\in\cK'$, are pairwise disjoint.
We shall bound $|\cK'|$ from below by overcounting the vertices excluded from $\cK'$ in every step $i$ of this process, 
i.e., those vertices $S\in \cK^{i-1}$ such that the neighborhoods of $S$ and $S_i$ have a non-empty intersection.
Recall that $N=(2+c)n$, $s=(1-c)n$, and $t=(1+2c)n$ for $c=0.02$. By~(\ref{eq:QnEH:stirling}), %we know that 
\begin{equation}
|\cN_t(S_i)|=\binom{N-s}{t-s}=\binom{(1+2 c )n}{3cn}\le \left(\frac{1.04^{1.04}}{0.06^{0.06} \cdot 0.98^{0.98}}\right)^n\le 1.26^{n}.
\label{eq:NtS}
\end{equation}

Similarly, there are at most $1.26^{n}$ vertices $S\in \cL_s$ such that $S\subseteq T$ for each $T\in \cN_t(S_i)$.
Thus, there are at most $1.26^{2n}$ vertices $S$ in $\cL_s$ such that $\cN_t(S)\cap \cN_t(S_1)\neq \varnothing$, see also Figure~\ref{fig:QnEH_NtS}.
In particular, $|\cK^{i}\setminus \cK^{i-1}|\le 1.26^{2n}$, independently of $i$.
Using that $|\cK\cap\cS'|\ge 1.66^n$,
we can bound the number of steps in the greedy process from below by
$$
|\cK'| \ge \frac{|\cK\cap\cS'|}{1.26^{2n}} \ge \frac{1.66^n}{1.26^{2n}}  \ge  \left(\frac{1.66}{1.26^2}\right)^{n} \ge   1.04^{n}.
$$
\begin{figure}[h]
\centering
\includegraphics[scale=0.62]{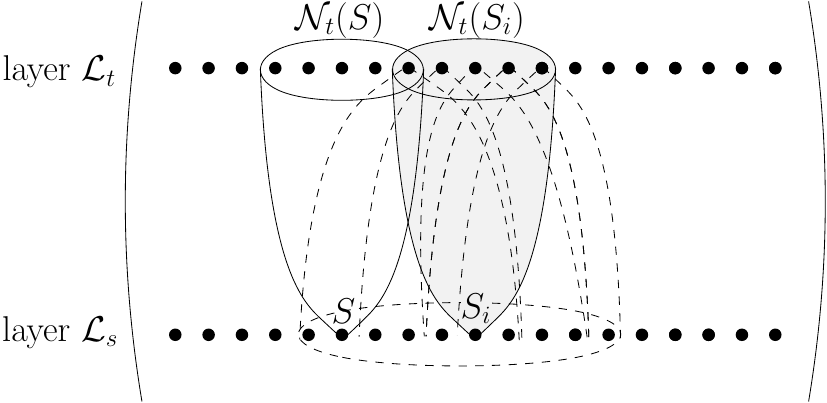}
\caption{Vertex $S\in\cL_s$ for which the neighborhood $\cN_t(S)$ intersects $\cN_t(S_i)$.}
\label{fig:QnEH_NtS}
\end{figure}

Our goal is to bound the probability that the cone $\cK$ is bad.
If $\cK$ is bad, then in particular $\cK'$ is bad, so
$$\PPP( \cK \text{ is bad}) \le  \PPP( \cK' \text{ is bad})=\PPP(\text{for any }S\in\cK'\text{,}\:\:\cN_t(S)\cap \cT\neq\varnothing),$$
where we used that $\cK'\subseteq\cS'$.
We defined $\cK'$ such that the neighborhoods $\cN_t(S)$, $S\in \cK'$, are pairwise disjoint.
In particular, the probability that a vertex $T\in \cN_t(S)$ is included in $\cT$ is independent of every $T'\in \cN_t(S')$, $S'\in \cK'$.
Thus,
$$
\PPP( \cK' \text{ is bad}) = \prod_{S\in \cK'} \PPP(\cN_t(S)\cap \cT\neq\varnothing).
$$
Next, we bound $\PPP(\cN_t(S)\cap \cT\neq\varnothing)$ for any fixed $S\in\cK'$. By (\ref{eq:NtS}),
$$
\PPP(\cN_t(S)\cap \cT\neq\varnothing)  \le  \sum_{T\in \cN_t(S)} \PPP(T\in \cT) = |\cN_t(S)| \cdot p  \le  \left(1.26 \cdot 0.77\right)^{n} \le  0.98^{n}.
$$
So,
$$
\PPP( \cK \text{ is bad}) \le \PPP( \cK' \text{ is bad}) \le  \prod_{S\in \cK'}  0.98^{n} =  0.98^{n|\cK'|} \le 0.98^{n\left(1.04\right)^{n}}.
$$

\noindent \textbf{Claim 4:} With high probability, there is no bad cone $\cK\in\text{Cone}_s$. \smallskip\\
\textit{Proof of Claim 4.}
By Claim 1, we have that with high probability, $|\cK\cap\cS'|\ge 1.66^n$ for every $\cK\in\text{Cone}_s$. From now on, suppose that $\cS'$ has this property.
Let $X_\text{bad}$ be the random variable counting the number of bad $\cK\in\text{Cone}_s$.
By Claim 3, the expected value of $X_\text{bad}$ is
\begin{eqnarray*}
\EEE(X_\text{bad}) & = & \sum_{\cK\in\text{Cone}_s} \PPP( \cK \text{ is bad})\\
&\le & \sum_{\substack{B\subseteq[N],\\ |B|=n}} \ \  \sum _{\substack{A\subseteq[N]\setminus B,\\ |A|=n/2}} \PPP( \cK \text{ is bad}) \\
&\le &2^{2N} 0.98^{n\left(1.04\right)^{n}} \\
&\le &2^{4.04n} 0.98^{n\left(1.04\right)^{n}} \to 0\text{ for }n\to \infty,
\end{eqnarray*}
thus, by Markov's inequality, $\PPP(X_\text{bad} \ge 1) \to 0$, and so, w.h.p., $X_\text{bad}=0$.
In particular, w.h.p., both conditions ~~ $|\cK\cap\cS'|\ge 1.66^n$ for every $\cK\in\text{Cone}_s$ ~~ and ~~ $X_\text{bad}=0$ ~~ are fulfilled, which proves Claim 4.
\\

By Claims 2 and 4, we know that w.h.p., for the randomly selected families $\cS'\subseteq\cL_s$ and $\cT\subseteq \cL_t$, there exists no bad cone in $\text{Cone}_s$,
and for every $\cK\in\text{Cone}_t$,\:\:$\cK\cap\cT\neq \varnothing$. This implies in particular the existence of two families $\cS'$ and $\cT$ with these properties.

For such fixed $\cS'$ and $\cT$, we refine the family $\cS'$ as follows.
Let $\cS$ be obtained from $\cS'$ by deleting all vertices $S\in\cS'$ for which $\cN_t(S)\cap \cT \neq\varnothing$, i.e., for which there is a $T\in\cT$ such that $S\subseteq T$.
By construction, $\cS$ and $\cT$ possess properties (i') and (ii').
Since there is no bad $\cK\in\text{Cone}_s$, there exists an $S\in\cK\cap\cS'$, for which the intersection $\cN_t(S)\cap \cT$ is non-empty.
Using the definition of $\cS$, we know that $S\in \cS$, thus $\cS$ has property (iii').
Furthermore, $\cT$ has property (iv').
Therefore, the families $\cS$ and $\cT$ are as desired.
\end{proof}
%\bigskip

\begin{construction}\label{constr:EHchain}
Let $n$ and $N$ be integers such that $N\ge 2n$.
Let $\cS$ and $\cT$ be two families of vertices in $\QQ([N])$ such that for every $S\in\cS$ and $T\in\cT$, it holds that $|S|<|T|$ and $S\not \subseteq T$.
We define a blue/red coloring of the Boolean lattice $\QQ([N])$.

Let $\cV_\cT$ be the set of all vertices $Z\in\QQ([N])$ with $|Z|\ge \tfrac{n}{2}$ such that there exists a $T\in\cT$ with $Z\subseteq T$.
Similarly, let $\cV_\cS$ be the set of all vertices $Z\in\QQ([N])$ for which $|Z|\le N-\tfrac{n}{2}$ and there is an $S\in\cS$ with $Z\supseteq S$.
Observe that $\cV_\cT$ and $\cV_\cS$ are disjoint, since the vertices of $\cS$ and $\cT$ are pairwise incomparable.
Let $\cW_\cS$ be the set of vertices $Z\in\QQ([N])$ with $\tfrac{n}{2}\le|Z|\le \tfrac{N}{2}$ and $Z\notin \cV_\cS$.
Similarly, let $\cW_\cT$ be the set of vertices $Z\in\QQ([N])$ for which $\tfrac{N}{2}<|Z|\le N-\tfrac{n}{2}$ and $Z\notin \cV_\cT$.

As illustrated in Figure \ref{fig:EHVTWT}, we color $Z\in\QQ([N])$ in
\begin{itemize}
\item blue if $|Z|< \tfrac{n}{2}$,
\item red if $Z\in \cV_\cT\cup\cW_\cS$,
\item blue if $Z\in \cV_\cS\cup\cW_\cT$,
\item red if $|Z|> N-\tfrac{n}{2}$.
\end{itemize}
\end{construction}

\noindent Note that this construction is well-defined if and only if $\cS$ and $\cT$ are element-wise incomparable.

\begin{figure}[h]
\centering
\includegraphics[scale=0.62]{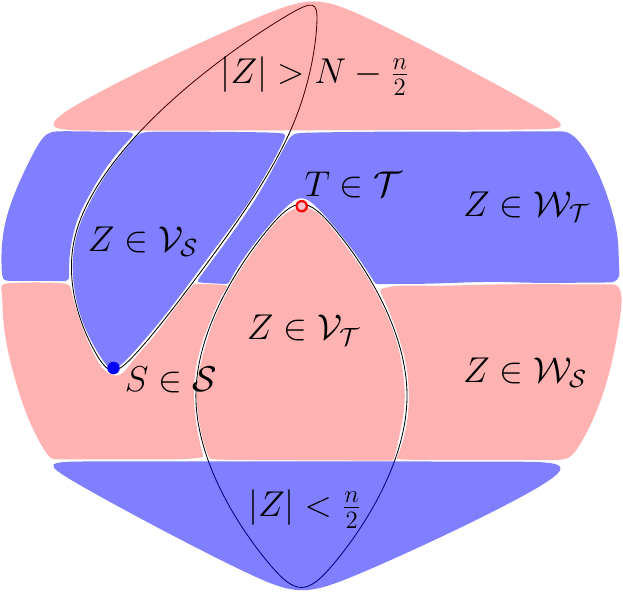}
\caption{Vertices in the sets $\cV_\cT, \cV_\cS, \cW_\cT$ and $\cW_\cS$ for exemplary $S\in \cS$ and $T\in\cT$.}
\label{fig:EHVTWT}
\end{figure}

\begin{proof}[Proof of Lemma \ref{lem:EHchain3}]
Let $ c =0.02$, and let $N=(2+ c )n$ for sufficiently large $n$. 
Let~$\cS$ and~$\cT$ be two families with properties as described in Lemma \ref{lem:EHchain_main}.
Color the Boolean lattice $\QQ([N])$ as defined in Construction \ref{constr:EHchain}, and let $\cV_\cS$ and $\cV_\cT$ as in Construction~\ref{constr:EHchain}.
It is easy to see that this coloring is $\dot C^{(rbr)}_{4}$-free, by using the observation that for every two vertices $A,B\in \cV_\cT\cup\cW_\cS$ with $A\subseteq B$, the subposet $\{Z\in\QQ([N]) : ~ A\subseteq Z \subseteq B\}$ is red.
We shall show that there is no monochromatic copy of $Q_n$, which implies that $\widetilde{R}(\dot C^{(rbr)}_{4},Q_n)> N=2.02n$.

%First, we verify that there is no red copy of $Q_n$ in $\QQ([N])$.
Assume towards a contradiction that there exists a red copy $\QQ'$ of $Q_n$ in $\QQ([N])$. %As a first step, we show that there is another red copy of $Q_n$ with additional properties.
By Lemma \ref{lem:embed}, there is an $n$-element $\bX\subseteq[N]$ such that $\QQ'$ is the image of an embedding
$\phi\colon \QQ(\bX)\to\QQ([N])$ such that $\phi(X)\cap\bX=X$ for every $X\subseteq\bX$.
Note that $|\phi(\varnothing)|\ge n/2$, because $\phi(\varnothing)$ is red.
Let $A$ be an arbitrary subset of $\phi(\varnothing)$ of size $|A|=n/2$, see Figure \ref{fig:QnEH_phiS}. 
Since $\phi(\varnothing)\cap \bX=\varnothing$, the subsets $A$ and $\bX$ are disjoint.
By property (iii) in Lemma \ref{lem:EHchain_main}, we know that there exists an $S\in\cS$ with $S \subseteq A\cup \bX$ and $|S\cap \bX|\le n/2$. 

\begin{figure}[h]
\centering
\includegraphics[scale=0.62]{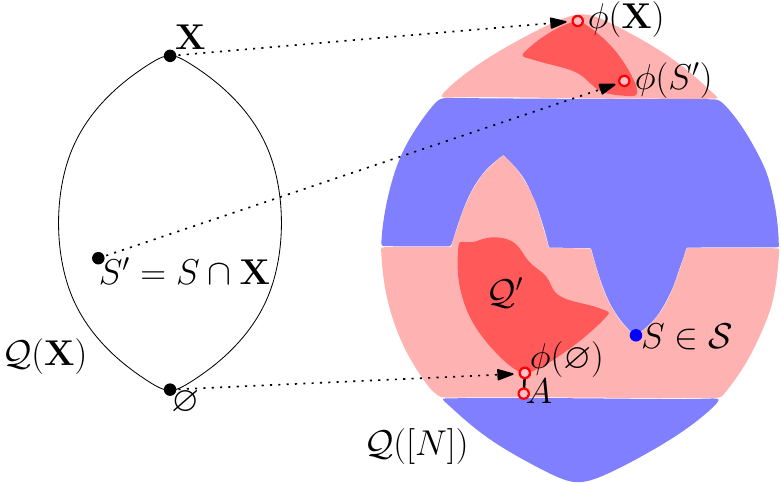}
\caption{Embedding $\phi$ of $\QQ(\bX)$ into $\QQ[N]$.}
\label{fig:QnEH_phiS}
\end{figure}

Let $S'=S\cap \bX$. We analyze $\phi(S')$ to find a contradiction.
First, we claim that $S\subseteq \phi(S')$. Indeed, using that $\phi$ is an embedding, we know that $S\cap A\subseteq A\subseteq \phi(\varnothing) \subseteq \phi(S')$.
Moreover, recall that $\phi(X)\cap \bX=X$ for all $X\subseteq \bX$, so $S'\subseteq \phi(S')$. Therefore, $S=(S\cap A)\cup S'\subseteq \phi(S')$. 
Because $S\in\cS$ and $S\subseteq \phi(S')$, either $\phi(S')\in\cV_\cS$ or $|\phi(S')|> N-\tfrac{n}{2}$.
Recall that $\phi(S')$ is a vertex in the red poset $\QQ'$, but every vertex in $\cV_\cS$ is blue.
This implies that $\phi(S')\notin\cV_\cS$, so $|\phi(S')|>N-\tfrac{n}{2}$.
However, because $\phi$ has the property that $\phi(X)\cap \bX=X$ for all $X\subseteq \bX$, $\phi(S')\cap (\bX\setminus S')=\varnothing$, so 
$$|\phi(S')|\le N - |\bX\setminus S'|=N - |\bX|+|S\cap \bX|\le N-\tfrac{n}{2},$$ a contradiction.
By a symmetric argument, there exists no blue copy of $Q_n$. Therefore, $\widetilde{R}(\dot C^{(rbr)}_{4},Q_n)>N$.
\end{proof}

In particular, we find that $R(Q_n,Q_n)\ge \widetilde{R}(\dot C^{(rbr)}_{4},Q_n)>2.02n$.
We remark that with the here presented approach it is not possible to push the lower bound on $R(Q_n,Q_n)$ higher than $\widetilde{R}(\dot C^{(rbr)}_{4},Q_n)$, i.e., higher than $3n$.
\bigskip

\noindent \textbf{Acknowledgments:}~\quad  The author would like to thank Maria Axenovich for proposing this problem and for helpful comments. The author thanks Ryan Martin, Torsten Ueckerdt, and two anonymous reviewers for valuable feedback on this manuscript. The author thanks Felix Clemen for discussions on the probabilistic arguments of this manuscript, and all participants of the \textit{Erd\H{o}s Center Workshop on Saturation in the Forbidden Subposet Problem} for many insightful sessions discussing related Ramsey-type questions.
\\

\noindent\textbf{Funding Statement:}~\quad Research was partially supported by DFG grant FKZ AX 93/2-1. The funder had no role in study design, data collection and analysis, decision to publish, or preparation of the manuscript.
\\

\noindent\textbf{Data Availability Statement:}~\quad Data sharing not applicable to this article as no datasets were generated or analysed during the current study.
\\

\noindent\textbf{Competing Interests Statement:}~\quad The author declares no competing interests.
\\

%The author would like to thank Maria Axenovich for many helpful comments and discussions, and her contribution to other parts of this series.
%The author thanks Torsten Ueckerdt for his comments leading to a cleaner argument for Lemma \ref{lem:count_perm}, Felix Clemen for valuable comments on the manuscript as well as the two anonymous referees of a previous version of this manuscript for their careful reading and their constructive feedback.
%\\

%\noindent \textbf{Competing interests statement:}~\quad  The author has no competing interests to declare.
%%%%%%%%%%%%%%%%%%%%%%%%

\end{document}